\documentclass[11pt]{elsart1p}

\let\theoremstyle\nothing

\usepackage{amsthm, amsmath, amssymb, pstricks}
\usepackage{graphicx}

\newtheorem{theorem}{Theorem}[section]

\newtheorem{corollary}[theorem]{Corollary}
\theoremstyle{definition}

\newtheorem{proposition}[theorem]{Proposition}
\theoremstyle{remark}
\newtheorem{remark}[theorem]{Remark}

\makeatletter
\@addtoreset{equation}{section}
\makeatother

\newcommand{\nn}{\mathbb{N}}

\newcommand{\rr}{\mathbb{R}}

\newcommand{\bfb}{\mathbf{b}}
\newcommand{\bfc}{\mathbf{c}}

\newcommand{\inD}[1][\relax]{\def\argone{#1}\def\temprelax{\relax}
  \ifx\argone\temprelax\right.\else\,\middle|#1\right.{}\fi}

\begin{document}

\begin{frontmatter}


\title{Eulerian-Catalan Numbers}

\author{Hoda Bidkhori}
\address{Department of Mathematics \\
North Carolina State University, Raleigh, NC, 27695 \\
{\tt hbidkho@ncsu.edu} }

\author{Seth Sullivant}
\address{Department of Mathematics \\
North Carolina State University, Raleigh, NC, 27695 \\
{\tt smsulli2@ncsu.edu}}

\begin{abstract}
We show that the Eulerian-Catalan numbers enumerate Dyck permutations.   We provide two proofs for this fact, the first using the geometry of alcoved polytopes and the second a direct combinatorial proof via an Eulerian-Catalan analogue of the Chung-Feller theorem.
\end{abstract}

\end{frontmatter}

\section{Introduction}

Let $A_{m,n}$ denote the \emph{Eulerian numbers}, which count the number of permutations on $n$ letters with $m$ descents.  The \emph{Eulerian-Catalan numbers} are defined by
$$
EC_{n}  =  \frac{1}{n+1} A_{n,2n+1}.
$$
We choose to attach the name Catalan to these numbers since $A_{n,2n+1}$ is the central Eulerian number, and for their connection to the Catalan numbers, which will become apparent shortly.  The Eulerian-Catalan numbers appear in the Online Encyclopedia of Integer Sequences \cite{Bagula}, however, no combinatorial interpretation appears there and we could not find one in the literature.

The Eulerian-Catalan number is clearly always an integer since  the Eulerian numbers satisfy the following relations
$$A_{m,n} = (n-m)A_{m-1,n-1} + (m+1)A_{m,n-1} \mbox{ and } A_{m,n} = A_{n-m -1, n} \mbox{ for all } m,n $$
which imply that 
$EC_{n} = A_{n-1,2n} + A_{n, 2n} = 2A_{n,2n}$.  
Given a permutation $w$ of $[n]$, we associate a $0/1$ sequence of length $n-1$, $ad(w)$, where 
$ad(w)_{i} = 0 $ if $w_{i} < w_{i+1}$ and $ad(w)_{i} = 1$ if $w_{i} > w_{i+1}$.  We call $ad(w)$ the ascent/descent vector of $w$.  A $0/1$ sequence is called a \emph{ballot sequence} if every initial string has at least as many zeroes as ones.

Alternately, the permutation $w$ defines a lattice path $L(w)$ starting from $(0,0)$ and with step $(1,0)$ if $i$ is an ascent, and with step $(0,1)$ for a descent.  Writing the entries of $w$ along the vertices of the path produces a standard young tableau of a border strip.  We call a permutation $w \in S_{2n+1}$ a \emph{Dyck permutation} if and only if $L(w)$ is a lattice path from $(0,0)$ to $(n, n)$, with all points on the path satisfying $y \leq x$.  By the usual correspondence between ballot sequences and Dyck paths (see, \cite[Ex. 6.19]{E1}), a permutation $w$ is a Dyck permutation if and only if $ad(w)$ is a ballot sequence.

Let $L$ be a lattice path from $(0,0)$ to $(n,n)$ using steps of $(0,1)$ and $(1,0)$.  The exceedance of $L$, denoted ${\rm exc}(L)$ is defined to be the number of $i \in \{0, \ldots, n\}$ such that there is a point $(i,i')$ in $L$ with $i < i'$.  Hence, the Dyck paths are the lattice paths with exceedance $0$.  The main results of this paper is the following:
\medskip

\begin{theorem}\label{main}
Fix $j = 0, \ldots, n$.  The number of permutations  $w \in S_{2n+1}$ with $n$ descents such that ${\rm exc}(L(w)) = j$ does not depend on $j$.
As a consequence the number of Dyck permutations $w \in S_{2n+1}$ is the Eulerian-Catalan number $EC_{n}.$
\end{theorem}
\vspace{4 mm}
We provide a direct combinatorial proof of Theorem \ref{main} in Section \ref{sec:direct}.  We also provide a geometric proof of the fact that the number of Dyck permutations $w \in S_{2n+1}$ is the Eulerian-Catalan number $EC_{n}$.  We define a polytope $P_{2,n}$ and we show that its normalized volume is the Eulerian-Catalan number.
 This is proved in Section \ref{pkn}.
The polytope $P_{2,n}$ turns out to be an alcoved polytope \cite{AP}, and hence its volume can also be interpreted as counting permutations with certain restrictions on its descent positions, which is explained in Section \ref{alcove}.  Combining these two arguments yields  the result that
the number of Dyck permutations $w \in S_{2n+1}$ is the Eulerian-Catalan number $EC_{n}$.  We prove these results in a Fuss-Catalan generality which gives a combinatorial interpretation for the numbers $\frac{1}{n+1} A_{n, kn + k -1}$ as the number of $(k-1)$-Dyck permutations.


\section{Subdividing the Hypersimplex}  \label{pkn}

The hypersimplex $\Delta(k,n)$ is the polytope
$$
\Delta (k,n) =  \left\{ (x_1, \ldots, x_n) \in [0,1]^n : \sum_{i = 1}^{n} x_{i} = k \right\}.
$$
It is well-known that the normalized volume of the hypersimplex  is the Eulerian number $A_{k-1,n-1}$.  Stanley~\cite{Sta1} provides a combinatorial proof of this fact by 
triangulating the hypersimplex.

Fix $k, n \in \nn$, and consider the hypersimplex $\Delta(n+1, k(n+1))$.
We define the polytope $P_{k, n}$  with the following inequalities: 
$$
P_{k,n} =  \left\{ (x_{1}, \ldots, x_{k(n+1)}) \in \Delta(n+1, k(n+1)) :  \sum_{s =1}^{kt} x_{ s} \leq t,  \quad t = 1, \ldots, n \right\}.
$$ 

\medskip

\begin{remark}
A $0/1$ sequence is called a \emph{$k$-ballot sequence} if every initial string has at least $k$-times as many $0$'s as $1$'s.  Note that a $1$-ballot sequence is an ordinary ballot sequence.   The polytope $P_{k,n}$ is equal to the convex hull of the $(k-1)$-ballot sequences of length $k(n+2)$.  This is shown in work of the first author \cite{HB}, where these polytopes are studied in the larger context of lattice path matroid polytopes.  Lattice path matroids were introduced in \cite{B1} and the Catalan matroid \cite{Fed} is a special case.  The polytope $P_{2,n}$ is the Catalan matroid polytope.  Corollary \ref{cor:volume} below implies that the normalized volume of the Catalan matroid polytope is the Eulerian-Catalan number.  We do not need these details here, and refer the reader to \cite{HB}.
\end{remark}

\medskip

For each $i \in \{0, \ldots, n\}$ define the polytope $P_{k,n,i} \subseteq \Delta(n+1, k(n+1))$ by the inequalities
$$
P_{k,n,i} =  \left\{ (x_{1}, \ldots, x_{k(n+1)}) \in \Delta(n+1, k(n+1)) :  \sum_{s =1}^{kt} x_{ki + s} \leq t,  \quad t = 1, \ldots, n \right\}
$$ 
where the indices are considered modulo $k(n+1)$.  For example, with $k = 2$ and $n = 2$, we get three polytopes:

$$
P_{2,2,0} =  \{ (x_{1}, \ldots, x_{6}) \in \Delta(3, 6)  :  x_{1} + x_{2} \leq 1, \quad  x_{1} + x_{2}+ x_{3} + x_{4} \leq 2 \}
$$
$$
P_{2,2,1} =  \{ (x_{1}, \ldots, x_{6}) \in \Delta(3, 6)  :  x_{3} + x_{4} \leq 1, \quad  x_{3} + x_{4}+ x_{5} + x_{6} \leq 2 \}
$$   
$$
P_{2,2,2} =  \{ (x_{1}, \ldots, x_{6}) \in \Delta(3, 6)  :  x_{5} + x_{6} \leq 1, \quad  x_{5} + x_{6}+ x_{1} + x_{2} \leq 2 \}.
$$
Note that $P_{k,n,0} = P_{k,n}$. 

\vspace{5 mm}

\begin{theorem}
Fix $k,n \in \nn$.  The interiors of the polytopes $P_{k,n,i}$ and $P_{k,n,j}$ are disjoint if $i \neq j$, and $\Delta(n+ 1, k(n+1)) = \cup_{i = 0}^{n} P_{k,n,i}$. 
\end{theorem}

\begin{proof}
Since the $P_{k,n,i}$ are all affinely isomorphic via a transformation that permutes coordinates, it suffices to show that an $x \in {\rm int} \, P_{k,n,0}$ does not belong to any other $P_{k,n,i}$.  Since $x \in {\rm int} \, P_{k,n,0}$ it satisfies the inequality
$$
x_{1} +  \cdots  + x_{ki}  < i.
$$
Since we are in the hypersimplex, we always have
$$
x_{1} + \cdots +  x_{k(n+1)}  = n+1.
$$
Combining these, we deduce that
$$
x_{ki +1} + \cdots  + x_{k(n+1)} > n+1 -i 
$$
which implies $x \notin P_{k,n,i}$.

\medskip

To prove that $\Delta(n+ 1, k(n+1)) = \cup_{i = 1}^{n+1} P_{k,n,i}$, we must show that any point $x \in \Delta(n+ 1, k(n+1))$ belongs to one of the $P_{k,n,i}$.  To do this, we consider the linear transformation
$$
\phi:  \rr^{k(n+1)} \rightarrow \rr^{n+1}
$$
such that $y_{i}  = -1 + \sum_{s = ki+1}^{k(i+1)} x_{s}$ where the coordinates on $\rr^{n+1}$ are $y_{0}, \ldots, y_{n}$.
The image of the hypersimplex $\Delta(k(n+1), n+1)$ is the polytope
$$Q(k,n) = \{ y \in [-1,k-1]^{n+1} : y_{0} + \cdots + y_{n} = 0 \}$$
and the image of $P_{k,n,i}$ is the polytope  
$$
R_{k,n,i} =  \{y \in Q(k,n) : y_{i} + \cdots  + y_{i + t} \leq 0, t = 0, n-1 \}. 
$$
Note that a point $x \in \Delta(n+1, k(n+1))$ belongs to $P_{k,n,i}$ if and only if its image in $Q(k,n)$ belongs to $R_{k,n,i}$.  Furthermore, the argument in the preceding paragraph implies that the $R(n,k,i)$ have disjoint interiors.  Hence, it suffices to show that $Q(k,n) = \cup_{i =0}^{n} R_{k,n,i}$.  Since the inequalities $-1 \leq y_{i} \leq k-1$ are common to all the polytopes, it suffices to show that the plane
$$
H(n) =  \{ y \in \rr^{n+1}: y_{0} + \cdots + y_{n} = 0 \}
$$
can be decomposed is the union of the cones
$$C(n,i) =  \{y \in H(n) : y_{i} + \cdots + y_{i + t} \leq 0, t = 0, n-1 \}.$$
For $i = 0,\ldots, n$, let  $v_{i} = -e_{i-1} + e_{i}$.  The cones $C(n,i)$ are simplicial, and it is straightforward to see that the generators of $C(n,i)$ are  $\{v_{0}, \ldots, v_{n} \} \setminus \{v_{i}\}$.  The vectors $v_{0}, \ldots, v_{n}$ also span $H(n)$, and $v_{0}+ \cdots + v_{n} = 0$.  This implies that the union of the cones spanned by the sub-collections of $n$ vectors is all of $H(n)$, which is what we needed to show.
\end{proof}

\begin{corollary}\label{cor:volume}
The normalized volume of $P_{k,n,i}$ is $\frac{1}{n+1} A_{n,kn + k -1}$.
\end{corollary}

\begin{proof}
Since $\Delta(n+1, k(n+1)) = \cup_{i=0}^{n} P_{k,n,i}$ and their relative interiors are disjoint, we have that 
$${\rm vol}(\Delta(n+1, k(n+1))) =  \sum_{i =0}^{n} {\rm vol} \, P_{k,n,i}.$$
As the cyclic shift of coordinates which sends $P_{k,n,0}$ to $P_{k,n,i}$ is volume preserving, and ${\rm vol}\, (\Delta(n+1, k(n+1))) = A_{n,kn + k -1}$, we deduce the desired formula.
\end{proof}


\section{Alcoved Polytopes}\label{alcove}

Fix  integers $k,n \in \nn$.  For $0 \leq i < j \leq n$, let $b_{ij}, c_{ij}$ be integers, with $b_{ij} \leq c_{ij}$.  The alcoved polytope defined by this data is:
$$
P(k, n,\bfb, \bfc)  :=   \left\{ (x_{1}, \ldots, x_{n}) \in \rr^{n} : \sum_{i =1}^{n} x_{i} = k,  b_{ij} \leq x_{i+1} + \cdots + x_{j} \leq c_{ij}, 0 \leq i < j \leq n \right\}.
$$
Alcoved polytopes were studied in detail by Lam and Postnikov \cite{AP}.  Alcoved polytopes have natural triangulations into unimodular simplices which are themselves alcoved polytopes and which are indexed by permutations; these simplices are called alcoves.  In the special case,  where $P(k,n,\bfb,\bfc) \subseteq \Delta(k,n)$, Lam and Postnikov give an explicit description of the simplices involved in the alcove triangulation, and hence a combinatorial formula for the volume of these alcoved polytopes.

Let $W(k,n,\bfb,\bfc) \subset S_{n-1}$ be the set of permutations $w =
w_1w_2\cdots w_{n-1} \in S_{n-1}$ satisfying the following
conditions:
\begin{enumerate}
\item
$w$ has $k-1$ descents.
\item
The sequence $w_{i} \cdots w_{j}$ has at least $b_{ij}$ descents.
Furthermore, if $w_{i} \cdots w_{j}$ has exactly $b_{ij}$ descents,
then $w_{i} < w_{j}$.
\item
The sequence $w_{i} \cdots w_{j}$ has at most $c_{ij}$ descents.
Furthermore, if $w_{i} \cdots w_{j}$ has exactly $c_{ij}$ descents,
then we must have that $w_{i} > w_{j}$.
\end{enumerate}
In the above conditions we assume that $w_0 = 0$.

\vspace{4 mm}
\begin{theorem}\label{descent} \cite{AP}  The normalized volume of $P(k,n,\bfb, \bfc) \subseteq \Delta(k,n)$ is equal to $|W(k,n,\bfb,\bfc)|$.
\end{theorem}

\vspace{5mm}
We apply Theorem \ref{descent} to give a combinatorial formula for the volume of the polytope $P_{k,n}$. 
As discussed previously, a lattice path from $(0,0)$ to $(nk,n)$ using $(0,1)$ and $(1,0)$ steps is called a $k$-Dyck path if every point on the path satisfies $y \leq \frac{1}{k} x$.
When $L(w)$ is a $k$-Dyck path, $w$ is called a $k$-Dyck permutation.  A $0/1$ sequence is called a $k$-ballot sequence if each initial string has at least $k$ times as many $0$'s as $1$'s.  Note that $w$ is a $k$-Dyck permutation precisely where $ad(w)$ is a $k$-ballot sequence.

\vspace{5mm}
\begin{proposition} \label{prop:dyck}
The normalized volume of $P_{k,n}$ is equal to the number of permutations $w \in S_{kn+k -1}$ such that $L(w)$ is a $(k-1)$-Dyck path.  Equivalently, the volume equals the number of permutations $w \in S_{kn+k -1}$ such that ${\rm ad}(w)$ is a $(k-1)$-ballot sequence.
\end{proposition}  

\begin{proof}
The polytope $P_{k,n} \subset \Delta(n+1, k(n+1))$ is an alcoved polytope since its defining inequalities are $0 \leq x_{i} \leq 1$, $x_1+\cdots + x_{k(n+1)} = n+1$ and
$$\sum_{s =1}^{kt} x_{s} \leq t,  \quad t = 1, \ldots, n.
$$
Applying Theorem \ref{descent}, we see that the volume of $P_{k,n}$  is the number of permutations $w=w_1\cdots w_{kn + k -1}  \in S_{kn + k -1}$ satisfying the following conditions:
\begin{enumerate}
\item  $w$ has $n$ descents.
\item $w_1\cdots w_{ik}$ has at most $i-1$ descents, for $1 \leq i \leq n.$
\end{enumerate}
These conditions are satisfied if and only if $ad(w)$ is a $(k-1)$-ballot sequence.
\end{proof}

Combining Corollary \ref{cor:volume} and Proposition \ref{prop:dyck} we deduce:
\vspace{4 mm}
\begin{theorem}
The number of permutations $w \in S_{kn+k -1}$ such that $L(w)$ is a $(k-1)$-Dyck path  is $\frac{1}{n+1} A_{n,kn +k -1}$.  In particular, the number of Dyck permutations in $S_{2n+1}$ is the Eulerian-Catalan number $EC_{n}$.
\end{theorem}


\section{Exceedances of Lattice Paths and Eulerian-Catalan Numbers}\label{sec:direct}

In this section we prove Theorem \ref{main} which gives us a combinatorial interpretation of Eulerian-Catalan number in terms of certain permutation statistics.  This result  is related to a classic proof that the Catalan numbers enumerate Dyck paths \cite{Chung}.  First we provide a combinatorial proof of Theorem \ref{main}. Later, we give a geometric proof of  Theorem \ref{main} for  the case $j=1$ (and hence also for $j = n-1$).  In Sections \ref{pkn} and \ref{alcove}, we saw a geometric proof for the cases
$j=0$ (and hence also for $j = n$). It is an interesting problem to provide
geometric proofs for other cases.



\begin{proof}[Proof of Theorem \ref{main}]
Consider a lattice path $P$ with $(0,1)$ and $(1,0)$ steps from $(0,0)$ to $(n,n)$.  Define $c(P)=(c_0,\ldots, c_n)$, where 
$c_i$ is the number of horizontal steps of $P$ at height $y=i.$  The cyclic permutations
$C_j=(c_j,\ldots, c_{j-1})$ of $c(P)$ are all distinct, and for each there is a unique path $P_j$ from $(0,0)$ to $(n,n)$ so that
$c(P_j)=C_j$.  The number of exceedances  of the paths $P_0,\ldots, P_n$ are the numbers $0,1,\ldots,n$ in some 
order.  These results are known as the \emph{Chung-Feller Theorem} \cite{Chung}.
 
Consider a permutation $W=w_1\cdots w_{2n+1}$ with $n$ descents.  We have one of the following cases:
 \begin{enumerate}
 \item The cyclic permutation $(w_1\cdots w_{2n+1})$ has $n$ cyclic descents.
 \item The cyclic permutation $(w_1\cdots w_{2n+1})$ has $n+1$ cyclic descents.
 \end{enumerate}
The cyclic descents of a (cyclic) permutation also include the possibility of a descent at the last position $w_{2n+1}$, which occurs when $w_{1} < w_{2n+1}$.

In the case $(w_1\cdots w_{2n+1})$ has $n$ cyclic descents, we consider distinct indexes $1 = i_0 < \cdots < i_{n} \leq 2n+1$ so that $w_{i_{k} -1}w_{i_k}$ is not a cyclic descent in the cyclic permutation $(w_1\cdots w_{2n+1})$.  Now, consider the $n+1$ permutations obtain by cyclic shifting the permutation $W$,
starting at $w_{i_0},\ldots, w_{i_{n}}$. We denote these permutations by $W_0, \ldots, W_n$. The permutations $W_0,\ldots, W_n$ all have $n$ descents, whereas all other cyclic shiftings of $W$ have $n-1$ descents.   As in the Chung-Feller theorem, we  define 
$c(L(W_0))=(c_0,\ldots,c_n)$ where $c_i$ is the number of horizontal steps at height $i$ for a lattice path $L(W_0)$.  The cyclic shiftings of $c(L(W_0))$, $C_j=(c_j,\ldots, c_{j-1})$ are all distinct, and for each there is a unique path $P_j$ from $(0,0)$ to $(n,n)$ so that
$c(P_j)=C_j$. It is easy to see that $P_j= L(W_j)$. By the Chung-Feller theorem,
  the lattice path associated to these permutations have different number of exceedance $0,\ldots, n.$  This shows that the $n+1$ cyclic shiftings of permutation $W$ which have $n$ descents, have $0, \ldots, n$ number of exceedances in some order.
  
 \vspace{4 mm}

Now, consider the case where the cyclic permutation $(w_1\cdots w_{2n+1})$ has $n+1$ cyclic descents.  For any permutation $W=w_1\cdots w_{2n+1}$, we define: 
$\widehat{W}=(\widehat{w}_1\cdots, \widehat{w}_{2n+1})=(2n+2-w_1)\cdots(2n+2-w_{2n+1})$.  If $W$ has $k$ exceedances, $\widehat{W}$ has $n-k$ exceedances, and if $W$ has $k$ cyclic descents, $\widehat{W}$ has $2n+1 - k$ cyclic descents.    We consider $i_0, \ldots , i_{n}$ so that   $w_{i_{k}-1}w_{i_k}$ is a cyclic descent in  $W$ and therefore $\widehat{W}$ does not have a cyclic descent at $\widehat{w}_{i_{k}-1}\widehat{w}_{i_k}$.  We  consider the $n+1$ permutations obtained by cyclic shifting the permutation $W$,
starting with $w_{i_0},\ldots, w_{i_{n}}$.  We denote them by $W_0,\ldots, W_n$. As above, these are the only permutations obtained by cyclic shifting $W$ that have $n$  descents, all other shiftings having $n+1$ descents.  As we see in the first case, we know that  the lattice path associated to $L(\widehat{W}_j)$ have different number of exceedance for $j=0,\ldots,n$. Therefore,  $P_j= L(W_j)$ for $j=0,\ldots,n$ have all  different numbers of exceedances  from $0,\ldots,n$ in some order.

Combining the above two results,  we know that among  the $2n+1$ cyclic shifts of a permutation $w$ with $n$ descents,  exactly $n+1$ of them have $n$ descents, and their associated lattice paths have different numbers of exceedances $0,\ldots,n.$
Therefore, the number of permutations  $ w \in S_{2n+1}$ with $n$ descent  and ${\rm exc}(L(w))=j$ are the same for $j=0,\ldots,n$.  As the total number of permutations $w \in S_{2n+1}$ with $n$ descents is $A_{n,2n+1},$ we see that the number of permutations  $ w \in S_{2n+1}$ with $n$ descent  and ${\rm exc}(L(w))=j$ is the Eulerian-Catalan number.
\end{proof}

Note that the proof of Theorem \ref{main} also gives a direct combinatorial proof for the fact that the number of Dyck permutations in $S_{2n+1}$ is $A_{n-1,2n} + A_{n,2n}$.  Indeed, the Dyck permutations are in bijective correspondence with the set of cyclic permutation with either $n$ or $n+1$ cyclic descents.  Cycling such a permutation until $2n+1$ is at the end and deleting $2n+1$ yields a bijection with permutations in $S_{2n}$ with either $n-1$ or $n$ descents.

To conclude this section, we provide a geometric proof of Theorem \ref{main} in the case of $j =1$, in the spirit of the proofs from Sections \ref{pkn} and \ref{alcove}.


\begin{proof}[Proof of Theorem \ref{main} with $j =1$]
Consider the polytope $P_{2,n}$, which is given by inequalities:
$$
P_{2,n} =  \left\{ (x_{1}, \ldots, x_{2(n+1)}) \in \Delta(n+1, 2(n+1)) :  \sum_{s =1}^{2t} x_{ s} \leq t,  \quad t = 1, \ldots, n \right\}.
$$
We call the inequality $\sum_{s =1}^{2t} x_{ s} \leq t$ the $t$-th inequality.  Let $T \subseteq [n]$ and consider the polytope defined by flipping the $t$-th inequalities for $t \in T$:
$$
P_{2,n}(T) = \left\{ (x_{1}, \ldots, x_{2(n+1)}) \in \Delta(n+1, 2(n+1)) : \sum_{s =1}^{2t} x_{ s} \geq t, t \in T,  \sum_{s =1}^{2t} x_{ s} \leq t, t \in [n] \setminus T \right\}.  
$$ 
Applying Theorem \ref{descent}, we see that the volume of $P_{2,n}(T)$ is the number of permutations $w \in S_{2n+1}$ such that $L(w)$ has an exceedance in position $t-1$ if and only if   $t \in T$ (that is, there is a point $(t-1, s)$ in $L(w)$ with $s > t-1$).  Thus, to prove Theorem \ref{main} it suffices to show that
\begin{equation}\label{eq:excedancesum}
EC_{n}  =  \sum_{T \subseteq [n] : |T| = j}  {\rm vol} ( P_{2,n}(T)).
\end{equation}
We prove Eq.(\ref{eq:excedancesum}) in the case $j = 1$.  

To do this, we consider the linear transformation
$$
\phi:  \rr^{2(n+1)} \rightarrow \rr^{n+1}
$$
such that $y_{i}  = -1 + \sum_{s = 2i+1}^{2(i+1)} x_{s}$ where the coordinates on $\rr^{n+1}$ are $y_{0}, \ldots, y_{n}$.  The image of the hypersimplex $\Delta(n+1,2(n+1))$ is the following polytope
$$Q(2,n) = \{ y \in [-1,1]^{n+1} : y_{0} + \cdots + y_{n} = 0 \}$$
and the image of $P_{2,n}(T)$ is the polytope  
$$
R_{2,n}(T) =  \{y \in Q(2,n) : y_{0} + \cdots  + y_{t-1} \geq 0,  t \in T, y_{0} + \cdots  + y_{t-1} \leq 0, t \in [n] \setminus T  \}. 
$$
Note that a point $x \in \Delta(n+1,2(n+1))$ belongs to $P_{2,n}(T)$ if and only if its image in $Q(2,n)$ belongs to $R_{2,n}(T)$.  Furthermore, all the $P_{2,n}(T)$ and $Q_{2,n}(T)$ are disjoint.  If we can find a collection of volume preserving linear transformations $\tau_{1}, \ldots, \tau_{n}$ such that $\tau_{1}(R_{2,n}(\{1\})), \ldots, \tau_{n}(R_{2,n}(\{n\}))$ have disjoint interiors and such that
$$
R_{2,n}(\emptyset) =  \cup_{i \in [n]} \tau_{i}( R_{2,n}(\{i\})),
$$
we will prove the theorem.   This follows because it is straightforward to lift the linear transformations $\tau_{i}$ to $\rr^{2(n+1)}$ in a way that will yield a similar result for the $P_{2,n}$.  All our linear transformations will be coordinate permutations.  Thus, we can ignore the common inequalities, $-1 \leq y_{i} \leq 1$, and consider the same questions in the plane
$$
H(n) =  \{ y \in \rr^{n+1}: y_{0} + \cdots + y_{n} = 0 \}
$$
for the cones
$$
C_{n}(T)  =   \{y \in H(n) : y_{0} + \cdots y_{t-1} \geq 0, t \in T, y_{0} + \cdots y_{t-1} \leq 0, t \in [n] \setminus T \}.
$$

Each of the cones $C_{n}(T)$ is simplicial, generated by the set of rays
$$
e_{t-1} - e_{t}  :  t \in T,   -e_{t-1} + e_{t} :  t \in [n] \setminus T,
$$
where $e_{i}$ denotes a standard unit vector.

For $t \in [n]$, consider the linear transformation $\tau_{t}$ which sends $y_{i} \mapsto y_{n-t+1 +i}$, where indices are considered modulo $n+1$.
This volume preserving linear transformation sends the generators of $C_{n}(\{t\})$ to the vectors
$$
-e_{0} + e_{n} \mbox{ and }  -e_{t-1} + e_{t} :  t \in [n] \setminus T.
$$
All of these rays belong to $C_{n}(\emptyset)$, hence  $\tau_{t}(C_{n}(\{t\})) \subseteq C_{n}(\emptyset)$.  Furthermore, each of the cones $\tau_{t}(C_{n}(\{t\}))$ is generated by a facet of $C_{n}(\emptyset)$ together with the same interior ray $-e_{0} + e_{n}$.  Hence, the set of cones $\{ \tau_{t}(C_{n}(\{t\})), t = 1, \ldots, n \}$ form the facets of a polyhedral subdivision of $C_{n}(\emptyset)$ which completes the proof.
\end{proof}


\section{Further Directions}

Our results on Eulerian-Catalan numbers suggest a number of interesting problems.

\begin{enumerate}
\item  Both the Catalan numbers and the Eulerian numbers have numerous combinatorial interpretations.  Are there other interpretations of the Eulerian-Catalan numbers as enumerating objects that are counted by the Eulerian numbers where a certain statistic is a Catalan object?  \item Both the Catalan numbers and Eulerian numbers have $q$ (and $q,t$) analogues.  Do these extend to the Eulerian-Catalan numbers?  
\item Catalan numbers and Eulerian numbers have natural generalizations beyond the symmetric group (i.e. to other types).  Can these be extended to Eulerian-Catalan numbers?  
\item   Generalizing the geometric proof of Theorem \ref{main} to arbitrary $j$ suggests the existence of interesting polyhedral decompositions of the cone of positive roots of a Weyl group. 
\end{enumerate}

\section*{Acknowledgement}
\label{sec:acknowledgement}
Hoda Bidkhori was partially supported by the David and Lucille Packard Foundation.  Seth Sullivant was partially supported by the David and Lucille Packard Foundation and the US National Science Foundation (DMS 0954865).

\end{document}